%% file: Rate_of_Decay_proof.tex
\documentclass[letterpaper,reqno]{amsart}

\allowdisplaybreaks[1]
\pdfoutput=1

\usepackage{amssymb,mathrsfs,amsmath,amsfonts,verbatim, amsthm}
\usepackage{mathtools}
\usepackage{bm,bbm}

\usepackage[colorinlistoftodos,textsize=small]{todonotes}
\usepackage{xargs} 
\usepackage{xcolor}
	\newcommandx{\concern}[2][1=]{\todo[color = red!55!,#1]{Concern: #2}} 
	\newcommandx{\refq}[2][1=]{\todo[color = yellow!40!,#1,]{Reference: #2}} 
	\newcommandx{\wording}[2][1=]{\todo[color = purple!50!,#1,]{Wording: #2}} 
	\newcommandx{\meric}[2][1=]{\todo[color = green!25!,#1]{#2}}
	\newcommandx{\mere}[2][1=]{\todo[color = blue!25!,#1]{#2}}
	\newcommandx{\palak}[2][1=]{\todo[color = violet!40!,#1]{#2}}
	\newcommandx{\mike}[2][1=]{\todo[color = orange!50!,#1]{#2}}

\usepackage{scalerel}

\usepackage{enumerate}

\usepackage{tikz}
\usetikzlibrary{arrows.meta}

\usepackage{microtype}
\usepackage{stackengine}

\usepackage[colorlinks=true,allcolors=blue]{hyperref}

\newtheorem{theorem}	{Theorem}[section]
\newtheorem*{theorem*}  {Theorem}
\newtheorem{corollary}	[theorem]{Corollary}

\newtheorem{lemma}		[theorem]{Lemma}

\theoremstyle{definition}
\newtheorem{definition}	[theorem]{Definition}

\newtheorem{remark}		[theorem]{Remark}

\input{newcommands}

\numberwithin{equation}{section}

\begin{document}

\author[Palak Arora]{Palak Arora${}^1$}
\thanks{${}^1$Research supported by NSF grant DMS-2154494}
\address{University of Florida}
\email{palak.arora@ufl.edu}

\author[Meric Augat]{Meric Augat${}^2$}
\thanks{${}^2$Research supported by NSF grant DMS-2155033}
\address{University of South Florida}
\email{maugat@usf.edu}

\author[Michael T. Jury]{Michael T. Jury${}^3$}
\thanks{${}^3$Research supported by NSF grant DMS-2154494}
\address{University of Florida}
\email{mjury@ufl.edu}

\author[Meredith Sargent]{Meredith Sargent${}^4$}
\address{University of Manitoba}
\thanks{${}^4$Supported by the Pacific Institute for the Mathematical Sciences (PIMS). The research and findings may not reflect those of the Institute.}
\email{meredith.sargent@umanitoba.ca}

\title{An optimal approximation problem for free polynomials}

\begin{abstract}
Motivated by recent work on optimal approximation by polynomials in the unit disk, we consider the following noncommutative approximation problem: for a polynomial $f$ in $d$ freely noncommuting arguments, find a free polynomial $p_n$, of degree at most $n$, to minimize $c_n := \|p_nf-1\|^2$. (Here the norm is the $\ell^2$ norm on coefficients.) We show that $c_n\to 0$ if and only if $f$ is nonsingular in a certain nc domain (the row ball), and prove quantitative bounds. As an application, we obtain a new proof of the characterization of polynomials cyclic for the $d$-shift.
    \end{abstract}
    
\maketitle

\section{Introduction}

This paper concerns an approximation problem for polynnomials in noncommuting indeterminates. To state the problem, let $\bbx = \set{x_1,\dots, x_d}$ be a set of freely noncommuting indeterminates and write $\x = (x_1,\dots, x_d)$.
	We write  $\C\fralg{\bbx}$ for the \textbf{free algebra} in the indeterminates $x_1, \dots, x_d$ and its elements are \textbf{free polynomials}.  We write elements $p\in \C\fralg{\bbx}$ as $p(x) = \sum_{w\in \fax} p_w x^w$ (with $p_w\in \C$ nonzero for all but finitely many $w$). Here, $w$ is a ``word" $w=i_1 i_2\cdots i_m$ in the ``letters" ${1, \dots, d}$, and $x^w:= x_{i_1}x_{i_2}\cdots x_{i_m}$. In this case the length of the word is defined to be $m$, and the {\bf degree} of $p$ is defined to be the maximum length of words $w$ for which $p_w\neq 0$. 
	
	We equip $\C\fralg{\bbx}$ with the natural inner product
obtained by declaring the ``words'' in the letters $x_i$ to be an orthonormal system: 
	\[
		\ip{x_{i_1}\dots x_{i_n}}{x_{j_1}\dots x_{j_m}} := \begin{cases}
			1 & m=n \text{ and } i_\ell = j_\ell \text{ for } 1\leq \ell \leq n;\\
			0 & \text{otherwise}
		\end{cases}
	\]
Thus for $p=\sum_{w\in\fax} p_w x^w$, $q=\sum_{w\in\fax} q_w x^w$, we have
\[
\langle p,q\rangle = \sum_{w\in\fax} p_w \overline{q_w}
\]
and the associated $\ell^2$ norm on coefficients
\[
\norm{p}^2 = \sum_{w\in\fax} |p_w|^2.
\]

Given a polynomial $f\in\C\fralg{\bbx}$, for each $n$ there is a unique polynomial $p_n\in \C\fralg{\bbx}$, of degree at most $n$, which minimizes the norm 
\[
c_n:=\norm{p_nf-1}^2
\]
We call this $p_n$ the {\em ${n}^{th}$ optimal polynomial approximant} ({\em OPA}) of $f^{-1}$. 
The sequence $(c_n)$ is evidently non-increasing. We can ask two questions about the numbers $(c_n)$:

{\bf Question 1:} {\em Under what conditions on $f$ do we have $c_n\to 0$? }

{\bf Question 2:} {\em If $c_n\to 0$, how fast? (That is, give bounds on the rate of decay.)}

Our main theorem gives answers to Questions 1 and 2, in the more general setting of polynomials with matrix coefficients. (For precise definitions see Section~\ref{sec:preliminaries}). 

	\begin{theorem}\label{mainRODthm-repeat}
    Let $F\in M_k(\C)\otimes \C\fralg{\bbx}$ be a $k\times k$-matrix valued free polynomial, let $(P_n)$ denote the sequence of optimal polynomial approximants, and put 
    \[
    c_n:=\|P_n(x)F(x)-I\|_2^2.
    \]
    Then: \begin{enumerate}
        \item $c_n\to 0$ if and only if $F$ is nonsingular in the row ball.
        \item If $F$ is nonsingular in the row ball, then 
        \[
        c_n\lesssim \frac{1}{n^p}
        \]
        for some $p>0$ depending on $F$. In particular, if $F$ is a product of $\ell$ atomic factors, we can take $p=\frac{1}{3^{\ell-1}}$.
    \end{enumerate}
    \end{theorem}
    Let us describe the motivation for these questions. The first motivation comes from the problem of characterizing the so-called {\em cyclic} free polynomials. We say $f$ is {\em cyclic} if every free polynomial $q$ can be approximated arbitrarily well in the $\ell^2$ norm by polynomials of the form $pf$.  It was proved in \cite{jury2020noncommutative} that a free polynomial $f$ is cyclic if and only if it is nonsingular in the row ball, but the proof is quite abstract and indirect. It turns out that cyclicity of $f$ is equivalent to the statement that $c_n\to 0$ (see Section~\ref{sec:preliminaries}). In this light, Question 1 asks about ``qualitative cyclicity,'' and Question 2 asks about ``quantitative cyclicity.''  We refer to \cite{jury2020noncommutative} for further discussion and motivation for the cyclicity question. 
    
    Thus, one of our motivations was to try to see if there is a more elementary proof of the characterization of cyclic $f$ obtained in \cite{jury2020noncommutative}. It turns out that such a proof is possible, and in the proof we present in fact allows us to give some answer to Question 2 (namely, item (2) of the theorem), though we do not know if this answer is sharp (see the remarks at the end of Section~\ref{sec:proof-of-main}). 
    
    Besides this, the problem can be motivated by considering what is already known for polynomials in one complex variable, and asking to what extent this result can be generalized to the free setting. In particular, Beneteau et al. \cite{JAM15} considered the analogous approximation problem for polynomials in one complex variable: given a polynomial $f\in\C[z]$, let $p_n$ be the polynomial of degree at most $n$ minimizing the $\ell^2$ norm $c_n:=\|p_nf-1\|^2$, and estimate the $c_n$. Among other things, they proved the following (we have stated their result in a slightly different form):
    \begin{theorem}\label{thm:original} Let $f\in \C[z]$ be a polynomial with no zeroes in the disk $|z|<1$. Then: 
    \[
    c_n:=\|p_n f-1\|^2 \lesssim \frac{1}{n}.
    \]
    \end{theorem}

The study of optimal polynomial approximants has roots tracing back to at least the 1970's where engineering and applied mathematics researchers investigated polynomial least-squares inverses in the context of digital filters in signal processing.  A survey for mathematical audiences can be found in \cite{BCSurvey}.  Roughly speaking, one considers the polynomials $p_n$ as approximations to $1/f$, even though $1/f$ will not in general have $\ell^2$ coefficients and therefore cannot be approximated directly by Hilbert space methods. 

In the mathematical context, OPAs were introduced in \cite{JAM15} as a tool to study cyclicity in a family of Dirichlet-type spaces on the disk $\D$ which includes the classical Hardy space of the disk $H^2(\D)$. (This was then extended to a more general reproducing kernel Hilbert space context in \cite{FMS14}. Surveys of OPA results can be found in \cite{SecoSurvey, BCSurvey}.) A function $f$ is said to be cyclic (for the shift operator $Sf:=zf$) in $H^2(\D)$ if
\begin{equation}
	[f]=\overline{\text{span}\set{z^kf\::\:k=1,2,3,\dots}}
\end{equation}
is the entire space $H^2(\D)$. It is not hard to see that to be cyclic, the function $f$ must be zero-free in the disk, and if the polynomials are dense, the function $1$ is cyclic. One can then observe that for a cyclic function $f$, the optimal norms $\norm{p_nf-1}$ converge to zero. Because optimal approximants are unique, the rate of decay of these optimal norms can then be used to quantify and compare ``how cyclic'' different functions are, e.g. a fast rate of decay would mean a function is strongly cyclic.

The questions about rates of decay and locations of zeros for optimal approximants have been considered in several contexts beyond the spaces mentioned above. The sequence spaces $\ell^p_A$ are considered in \cite{MR4159768,MR4406262}, and \cite{MR4448804} considers $L^p$ and $H^p$, and the pairs of papers \cite{PJM15,TAMS16} and \cite{MSAS1,MSAS2} consider approximations problems adapted to the bidisk and unit ball, respectively. 

    The role of the hypothesis on the zeroes of $f$ in Theorem~\ref{thm:original} is explained by the following observation: note that polynomials $p$ and complex numbers $z$, the evaluation functional $p\to p(z)$ is continuous for the $\ell^2$ norm if and only if $|z|<1$. Thus, it quickly follows that a necessary condition for $c_n\to 0$ is that $f$ be nonvanishing in the disk $|z|<1$. It is a folklore theorem that this condition is also sufficient for $c_n\to 0$, but as far as we are aware this result from \cite{JAM15} is the first to obtain a quantitative bound. The result can be proved by explicit calculation in the affine linear case $f(z)=1-\mu z$ ($|\mu|\leq 1$), and then extended to general $f$ by factoring and induction (though some care is required in handling repeated roots). Thus, in the free case, one may suspect that some sort of nonvanishing condition on $f$ will be necessary for cyclicity, and indeed this is the case, for essentially the same reason that certain evaluaion functionals $f\to f(X)$ (now at matrix points $X$) are bounded for the $\ell^2$ norm. The details may be found in Section~\ref{sec:preliminaries}.

    Of course, in the free setting we will have no recourse to the fundamental theorem of algebra, so rather different techniques will be required. In particular, while any free polynomial $f$ can be factored into irreducibles, these irreducibles will not in general be linear (and indeed the factorization need not be unique, e.g. $x-xyx=x(1-yx)=(1-xy)x$). Nonetheless, it turns out we will be able to reduce the general question to the case of (affine) linear $f$, though at the cost of introducing matrix coefficients. The technique rests on the (by now) well-understood technique of {\em linearizations} or {\em realizations} of free polynomials and rational functions. 

   
    \subsection{Reader's Guide}
    The next section will give a short tour of our setting in the nc (noncommutative) universe, including the row ball and the Fock space, along with the definition of an nc function. 
    We also discuss the mechanics of the approximation problem in this universe, including the definitions of the d-shift and cyclicity, and some simple conditions.
    Our results hold in the general case of polynomials with matrix coefficients, so subsection \ref{sec:matrix-coefficients} gives background on these, while subsection \ref{sec:stabeq} sets out the lemmas used to decompose a polynomial. We round out the preliminaries with a discussion of the outer spectral radius and some related lemmas.
    
    Section \ref{sec:proof-of-main} contains the proof of Theorem \ref{mainRODthm-repeat}, as well as further questions and examples.

\section{Preliminaries}\label{sec:preliminaries}
\subsection{The Fock space $\cF_d$, nc functions and nc domains}
\begin{definition}
	Let $\bbx = \set{x_1,\dots, x_d}$ be a set of freely noncommuting indeterminates and write $\x = (x_1,\dots, x_d)$.
	We write  $\C\fralg{\bbx}$ for the \textbf{free algebra} in the indeterminates $x_1, \dots, x_d$ and its elements are \textbf{free polynomials}.
	We define an inner product on $\C\fralg{\bbx}$ by declaring the ``words'' in the letters $x_i$ to be an orthonormal system: 
	\[
		\ip{x_{i_1}\dots x_{i_n}}{x_{j_1}\dots x_{j_m}} := \begin{cases}
			1 & m=n \text{ and } i_\ell = j_\ell \text{ for } 1\leq \ell \leq n;\\
			0 & \text{otherwise}
		\end{cases}
	\]
	This inner product induces a norm $\norm{\cdot}$ and the completion of $\C\fralg{\bbx}$ in terms of this norm is $\cF_d$, the \textbf{full Fock space} 
	in $d$ letters. We may identify $\cF_d$ with the space of formal power series in the words $w$, with square-summable coefficients:
\[
	\cF_d = \set{\sum_{w\in \fax} a_w w \, : \, \sum_{w\in \fax} \abs{a_w}^2 < \infty }.
\]
\end{definition}
When $d=1$, every power series $\sum a_n z^n$ with square-summable coefficients has radius of convergence at least $1$, and hence defines a function $f(z)=\sum a_n z^n$ in the unit disk $|z|<1$. In our free setting, it turns out that the formal power series we are considering can also be viewed as {\em noncommutative (nc) function} on an appropriate {\em nc domain}. To make this precise we begin with some definitions. 

	The noncommutative analog of the disk $|z|<1$ will be the {\em row ball}:
	
	\begin{definition} Let $d\geq 1$ be an integer. For a $d$-tuple of $k\times k$ matrices $X=(X_1, \dots, X_d)$ we define the {\em row norm} of $X$ to be 
	\[
	\|X\|_{row}:= \|X_1X_1^*+\cdots +X_dX_d^*\|^{1/2}
	\]
	(that is, the usual operator norm of the $k\times dk$ matrix $(X_1 \ X_2\ \cdots X_d)$). 
	Similarly we define the {\em column norm}
	\[
	\|X\|_{col}:= \|X_1^*X_1+\cdots+ X_d^*X_d\|^{1/2}.
	\]
	We say $X$ is a {\em row contraction} if $\|X\|_{row}\leq 1$ and a {\em strict row contraction} if $\|X\|_{row}<1$. Column contraction and strict column contraction are defined similarly.

	For fixed $d$ and each $k\geq 1$, we denote
	\[
	\mfB^d_k = \{X\in M_k(\C)^d  : \|X\|_{row}<1\}
	\]
	and put 
	\[
	\mfB^d := \bigsqcup_{k=1}^\infty \mfB_k^d.
	\]
	The set $\mfB^d$ is called the {\em row ball}, and we refer to $\mfB^d_k$ as the $k^{th}$ ``level" of the row ball. 
	\end{definition}
	We note that each $\mfB^d_k$ is an open set in the usual topology on $M_k(\C)^d$, and $\mfB^d$ is closed under direct sums: if $X\in \mfB^d_{k_1}$ and $Y\in \mfB^d_{k_2}$ then $X\oplus Y := (X_1\oplus Y_1, \dots, X_d\oplus Y_d)\in \mfB^d_{k_1+k_2}$. The row ball is thus an {\em nc domain}. 
With these definitions, we have: 
\begin{theorem} \cite[Theorem 1.1]{MR2264252}
If $\sum_{w}|a_w|^2<\infty$ and $X$ is a $d$-tuple of $k\times k$ matrices with $\|X\|_{row}<1$, then the series
\[
f(X):=\sum_{w\in \fax} a_w X^w
\]
converges in the usual topology of $M_k(\C)$. 
\end{theorem}
Thus, every element of the Fock space $\cF_d$ determines a graded function on the row ball $f:\mfB^d_k\to M_k(\C)$. This function respects direct sums: $f(X\oplus Y)=f(X)\oplus f(Y)$ and similarities: if $S\in GL_k(\C)$ and both $\|X\|_{row}<1$, $\|S^{-1}XS\|_{row}<1$, then $f(S^{-1}XS)=S^{-1}f(X)S$. This $f$ is then an {\em nc function} on the row ball. (We remark that the theorem and subsequent remarks also hold with the column ball in place of the row ball). 

In particular, for any $X$ in the row ball at level $k$, the map from $\cF_d$ to $M_k(\C)$ given by
\[
f\to f(X)
\]
is continuous for the norm topologies on each space. (One thinks of this as ``bounded point evaluation'' analogous to the point evaluation $f\to f(z)$ for $f$ in the Hardy space and $|z|<1$.)  As a consquence we obtain the following proposition, relevant to our approximation problem: 
\begin{theorem}\label{nonsingular-is-necessary} If $f$ is an nc polynomial and $p_n$ is a sequence of nc polynomials such that
\[
\lim_{n\to \infty}\norm{p_nf-1}_2 =0,
\]
then $f$ is nonsingular in the row ball. 
\end{theorem}
\begin{proof}
Suppose $X$ is in the row ball and $\det f(X)=0$. Then also $\det(p_n(X)f(X))=0$ for all $n$, so each of the matrices $p_n(X)f(X)$ is singular, and hence these cannot converge to $I$. 

\end{proof}

\subsection{The $d$-shift and cyclicity} From now on we think of the Fock space $\cF_d$ in this way, as a space of nc functions represented as convergent powers series in the row ball, in analogy with spaces of holomorphic functions in the disk. Let us introduce an important class of operators acting in the space $\cF_d$:

\begin{definition}
The \textbf{left $d$-shift} is the tuple of operators $L = (L_1,\dots, L_d)$ where each $L_i:\cF_d\to \cF_d$ is given by $L_i:f\mapsto x_if$.
We similarly define the \textbf{right $d$-shift}.
\end{definition}

In the nc function picture, for each $f\in \cF_d$ and each $X\in \mfB^d$
\[
(L_if)(X) = X_i f(X).
\]
It is evident that each $L_i$ is a bounded operator for the $\ell^2$ norm on coefficients. In fact, it follows quickly from definitions that each of the operators $L_i$ is an isometry for the $\ell^2$ norm, and their ranges are mutually orthogonal. These facts are summarized algebraically in the relations
\[
L_i^*L_j = \delta{ij}I.
\]

For any polynomial $p\in \C\fralg{\bbx}$, the ``left multiplication'' operator
\[
f(X)\to p(X)f(X)
\]
is bounded, since this operator is just $p(L_1, \dots, L_d)$. Similarly the ``right multiplication operators'' $f(X)\to f(X)p(X)$ are bounded.

Analogous to the notion of cyclicity in Hardy space, we make the following definition:  
\begin{definition}\label{def:free-cyclic}
An element $f\in \cF_d$ is said to be  {\bf cyclic} for the $d$-shift if the set 
\(
\{ pf: p\in \C\fralg{\bbx}\}
\)
is dense in $\cF_d$. 
\end{definition}

For a given free polynomial $f$, it is evident that $f$ is cyclic if and only if $c_n:=\|p_nf-1\|^2\to 0$. Indeed, if $f$ is cyclic then there is some sequence of polynomials $q_n$ so that $\|q_nf-1\|^2\to 0$, so by the optimality of the $p_n$ we must have $\|p_nf -1\|^2\to 0$ as well. On the other hand, if $\|p_nf-1\|^2\to 0$ and $q$ is any polynomial, then since $q$ is a bounded left multiplier, we have $\|qp_nf-q\|^2\to 0$ as well, so that $f$ is cyclic.

\subsection{Polynomials with matrix coefficients}\label{sec:matrix-coefficients} We will also have cause to deal with matrix-valued versions of $\cF_d$. For fixed $m,n$ we view $M_{m\times n}(\C)$ as a Hilbert space with the tracial inner product $\langle A, B\rangle = tr(B^*A)$, so that $M_{m\times n}(\C)$ is a Hilbert space with the Hilbert-Schmidt or {\em Frobenius} norm. Elements of the Hilbert space tensor product $M_{m\times n}(\C)\otimes \cF_d$ may be identified with formal power series with $m\times n$ matrix coefficients
\[
\sum_{w\in \fax} A_w x^w
\]
with $\sum_{w\in \fax} tr(A_w^*A_w)<\infty$. These may in turn be identified with ``matrix-valued" functions on the row ball, where at each level $k$ for $\|X\|_{row}<1$ we have a convergent power series in $M_{m\times n}\otimes M_k$
\[
F(X): = \sum_{w\in\fax} A_w\otimes X^w.
\]
We again obtain bounded point evaluations $f\to f(X)$, and the analog of Theorem~\ref{nonsingular-is-necessary} holds in the (square) matrix-valued case as well. 

Also as before, each matrix-valued polynomial $P\in M_{m\times n}( \C\fralg{\bbx})$ will define bounded left and right multipliers $F\mapsto PF$, $F\mapsto FP$ between matrix-valued $\cF_d$ spaces of appropriate sizes. We will need the following lemma about multiplication by linear polynomials:
\begin{lemma}\label{lem:contractive-pencil}
    Let $A=(A_1, \dots, A_d)\in M_m(\C)^d$ let and $P(x):= A_1x_1+\cdots +A_dx_d$.  Then the left multiplication opeartor $F\mapsto PF$ has norm equal to $\|A\|_{col}$.
\end{lemma}
\begin{proof}
Writing $L_i$ for the left shifts as above, the operator $F\to PF$ acting in the Hilbert space tensor product $M_m\otimes \cF_d$ is given by $\sum_{j=1}^d A_j\otimes L_j$. (Here we identify $A_j$ with the operator $B\mapsto A_jB$ acting in the Hilbert space $M_m(\C)$.) From definitions, the operator $B\mapsto CB$ in the tracial Hilbert space $M_m(\C)$ has norm equal to $\|C\|_{op}$. Since the $L_j$ are isometries with orthogonal ranges, the relations $L_i^*L_j=\delta_{ij}$ entail
\begin{align*}
\norm{\brkt{\sum_{j=1}^d A_j\otimes L_j}}^2 
    &= \norm{\brkt{\sum_{j=1}^d A_j\otimes L_j}^*\brkt{\sum_{j=1}^d A_j\otimes L_j}} \\
    &=\norm{\sum_{j=1}^d (A_j^*A_j)\otimes I} 
        = \norm{\sum_{j=1}^d (A_j^*A_j)}_{op} \\
    &=\norm{A}_{col}^2.
\end{align*}

\end{proof}
	\vskip.2in

\subsection{Stable equivalence and linearization}\label{sec:stabeq}
\begin{definition}
An\, $m\times\ell$\, nc\, linear pencil (in $d$ indeterminates) is an expression of the form
\[L_A(\bbx)=A_0+A\bbx\] where $A\bbx=A_1x_1\,+\,\cdots\,+\,A_dx_d$ and $A_i\in M_{m\times\ell}(\C)$ for $i = 1, \dots, d$. If $A_0=I$ \,then\, we call \,$L_A(\bbx)$\, a monic linear pencil. A matrix tuple $A = (A_1,\dots, A_d)\in M_m(\C)^d$ is \textbf{irreducible} if $A$ generates $M_m(\C)$ as a (unital) algebra, i.e. 
	$M_m(\C) = \set{p(A) \, : \, p\in \C\fralg{\bbx}}$.
	The monic linear pencil $L_A$ is \textbf{irreducible} if $A$ is irreducible.
\end{definition}

The protagonist of our proof is stable associativity. Let us define that as well:
\begin{definition}
Given\, $A\in M_{k\times k}(\C\fralg{\bbx})$ and $B\in M_{\ell\times \ell}(\C\fralg{\bbx})$.\newline
We say $A$ and $B$ are \textbf{stably associated} if  there exists $\, N\in \ZZ^+$ and $P,Q\in \GL_N(\C\fralg{\bbx})$
\,	such\, that
	\[
		P\bpm A & \\ & I \epm Q = \bpm B & \\ & I \epm.
	\]\vspace{2mm}
	
We use the notation $A\sim_{\mathrm{sta}} B$ for $A$ being stable associated to $B$; it is evident that this is an equivalence relation. 
\end{definition}

\begin{definition}
If $f\in M_{k\times k}(\C\fralg{\bbx})$, then \[\sZ_n(f) = \set{X\in M_n(\C)^d \, : \, \det(f(X)) = 0},\] and \[\sZ(f) :=\bigsqcup_{n\geq1}\sZ_n(f).\]
	The set $\sZ(f)$ is the \textbf{free zero locus} of $f$.
\end{definition}
It is also evident from the definitions that if $f\sim_{sta} g$ then $\sZ(f)=\sZ(g)$. The concept of stable associativity is useful to us because of the following classical fact (we refer to \cite{HKVfactor} for a discussion of this in the context of factorization of free polynomials):
	\begin{lemma}[Linearization trick]\label{linearization trick}
		Suppose $F\in M_k(\C\fralg{\bbx})$ and $F(0) = I$. Then $F$ is stably associated to a monic linear pencil. Moreover, by Burnside's theorem we get that every monic pencil is similar to a pencil of the form
		\[ L_A(x)=
\begin{bmatrix}
 L_1(x) & * & * & \hdots & * \\
 0 & L_2(x) & * & \hdots & * \\
 0 & 0 & L_3(x) &  \hdots & * \\
 \vdots & \vdots &\vdots & \ddots  & \vdots\\
 0 & 0 & 0 & \hdots & L_l(x)
\end{bmatrix}
\]
where for every $k$, $L_k(x) = I-\sum_{j=1}^d \Bar{A}^{(k)}_jx_j = I-\Bar{A}^{(k)}x$ for  $\Bar{A}^{(k)}=(\Bar{A}^{(k)}_1\: \Bar{A}^{(k)}_2\: \cdots \: \Bar{A}^{(k)}_d)$ and either $\Bar{A}^{(k)}=0$ or irreducible.
\end{lemma}
		
For example, the $1\times 1$ polynomial $F(x,y)=1-xy$ is stably associated to the $2\times 2$ monic linear pencil $\bpm 1 & x \\ y & 1\epm$ via 
\[
		\bpm 1 & x \\ 0 & 1 \epm \bpm 1 - xy & 0 \\ 0 & 1 \epm \bpm 1 & 0 \\ y & 1 \epm = \bpm 1 & x \\y & 1 \epm.
	\]
(This calculation is sometimes known as ``Higman's trick.'') In fact the full lemma can be proved by iteratively applying this trick (increasing the size of the matrices at each step) to gradually reduce the degree of each monomial appearing in the entries of $F$. 


	
	We say $F\in M_k(\C\fralg{\bbx})$ is \textbf{regular} if it is not a zero divisor. In particular, if $F(0) = I$ then $F$ is regular.
	A regular non-invertible matrix polynomial $F$ is an \textbf{atom} if it is not the product of two non-invertible matrices in $M_k(\C\fralg{\bbx})$.
	When $k=1$, every nonzero polynomial is regular and a nonconstant polynomial $p$ is an atom if it is not the product of two nonconstant 
	polynomials. 
	If $F_1,F_2$ are regular matrices over $\C\fralg{\bbx}$ (not necessarily of the same size) and $F_1\sim_{\mathrm{sta}} F_2$, then $F_1$ 
	is an atom if and only if $F_2$ is an atom. The following lemma from \cite{HKVfactor} relates irreducibility of the polynomial $F$ to irreducibility of the pencil $L_A$. 
	


	\begin{lemma}\cite[Lemma 4.2]{HKVfactor}
		If $F\in M_k(\C\fralg{\bbx})$ is an atom and $F(0)=I$, then $F$ is stably associated to an irreducible monic linear pencil $L$.
		That is, if $F$ is an atom, then there exists an irreducible tuple $A\in M_m(\C)^d$ such that $F$ is stably associated to $L_A = I_m 
		- \sum_{i=1}^d A_ix_i$.
	\end{lemma}

	
	It is worth noting that in general, $L_A$ is not guaranteed to have a nonempty free zero locus.
	In fact, $\sZ(L_A) = \varnothing$ if and only if $A$ is a jointly nilpotent tuple (i.e. there exists $N\in \NN$ such that $w(A) = 0$ for 
	all $\abs{w}\geq N$).
	However, if $A$ is irreducible, then $A$ cannot be jointly nilpotent, thus $\sZ(L_A) \neq \varnothing$.

	If $F$ is stably associated to $L_A$, then for any invertible $S$ of the same size as $A$, $F$ is stably associated to $L_{S^{-1}AS}$.

The following lemma shows that our rate-of-decay problem is invariant under stable associativity. Precisely: 

	\begin{lemma}\label{same ROD}
		Suppose $F$ and $G$ are matrix polynomials.
		If $F$ and $G$ are stably associated, then there exists $C_1,\,C_2\in \RR$ and $D_1, D_2\in \NN$ such that for $P_n$,\,$Q_n$, the $n$th degree OPAs of $F$ and $G$ respectively, the following holds: 
	    \[
		C_1\norm{P_{n+D_1}F - I}_F^2 \leq \norm{Q_nG- I}_F^2 \leq C_2\norm{P_{n-D_2}F - I}_F^2.
	    \]
	\end{lemma}
		
	    

	\begin{proof} 

		Since $F$ and $G$ are stably associated, there exists $N\in \ZZ^+$ and $A,B\in \GL_N(\C\fralg{\bbx})$ such that
		\[
			(F\oplus I) = A(G\oplus I)B.
		\]
	Since $B$ and $B^{-1}$ are both polynomials, it follows that the map $H\mapsto BHB^{-1}$ is bounded in $\cF_d$, so there there exists $C_1\in \RR$ such that 
		$\norm{BHB^{-1}}_F\leq C_1\norm{H}_F$ for all $H\in M_N(\C\fralg{\bbx})$.
		Next, set $D_1 = \deg(A) + \deg(B)$ (here, $\deg(A)$ is the maximum of the degrees taken over the entries of $A$).
		
		Observe
		\begin{align*}
			B\bpm P_n & \\ & I \epm A \bpm G & \\ & I \epm - \bpm I & \\ & I \epm
				&= B\left(\bpm P_n & \\ & I \epm A \bpm G & \\ & I \epm B - \bpm I & \\ & I\epm\right)B^{-1} \\
				&= B\left(\bpm P_n & \\ & I \epm \bpm F & \\ & I \epm  - \bpm I & \\ & I\epm\right)B^{-1}.
		\end{align*}
		Hence,
		\begin{align*}
			\Norm{B\bpm P_n & \\ & I \epm A \bpm G & \\ & I \epm - \bpm I & \\ & I\epm}_F^2
				&\leq C\Norm{\bpm P_n & \\ & I \epm \bpm F & \\ & I \epm - \bpm I & \\ & I\epm}_F^2 \\
				&= C\Norm{P_nF - I}_F^2
		\end{align*}
		The matrix $B(P_n\oplus I)A$ has a block structure of $\bsbm Q & * \\ * & * \esbm$, and we note that the nature of the Frobenius norm 
		implies that
		\[
			\norm{QG-I}_F^2\leq \Norm{\bpm Q & * \\ * & * \epm \bpm G & \\ & I \epm - \bpm I & \\ & I \epm}_F^2.
		\]
		Thus,
		\[
			\norm{QG-I}_F^2\leq C\norm{P_nF-I}_F^2.
		\]
		Finally, note that 
		\[
			\deg(Q)\leq \deg(B(P_n\oplus I)A) \leq \deg(B) + \deg(P_n) + \deg(A) = \deg(P_n) + D_1.
		\]
		For $n$th degree OPA,  $Q_n$ of $G$, $\|Q_nG-I\|$ decreases with $n$, thus, we can write 
		\[\norm{Q_nG- I}_F^2 \leq C_1\norm{P_{n-D_1}F - I}_F^2.\]

		 
		 The other inequality follows by interchanging the roles of $F$ and $G$, and chasing through the proof, we find that in this case we can take $D_1=\deg(A^{-1})+\deg(B^{-1})$. 
	\end{proof}



\subsection{The Outer Spectral Radius}\label{sec:OuterSpecRad}
	The last technical tool we will need is the {\em outer spectral radius}, which is one of several possible notions of a ``joint spectral radius" for a system of matrices $A=(A_1, \dots, A_d)$. 
	\begin{definition}
	Let $n\in\mathbb{N}$ and $X\in M_{n\times n}(\C)^d$. We associate to $X$ the completely positive map on $M_{n\times n}$ defined by 
	\[
		\Psi_X(T) \,=\, \sum^d_{j=1}X_j T X_j^*.
	\]
    The \textbf{outer spectral radius} is defined to be $\rho(X) \,:=\,\lim_{k\to \infty} \|\Psi_X^k(I)\|^{1/2k}$ (that is, the square root of the spectral radius of $\Psi_X$ viewed as a linear transformation on $M_n$).
	\end{definition}
	
	\begin{remark}
	Observe that if we equip $M_n(\mathbb{C})$ with tracial inner product then $\Psi_X^*=\Psi_{X^*}$, so $\rho(X)=\rho(X^*)$. 
	\end{remark}
	There are two other equivalent definitions of outer spectral radius \cite{MR4344050}:

	Several properties of the outer spectral radius are pointed out in \cite{SSS20} Section 4;
	\begin{enumerate}
		\item $\rho(X) = \rho(S^{-1}XS)$ for any $S\in \GL_n(\C)$, follows from the definition $2.16$;
		\item $\rho(X)=\rho(X^*)$, follows from remark $2.17$;
		\item $\rho(X)<1$ if and only if $X$ is similar to an element of the row ball, this follows from Theorem $3.8$, \cite{PopescuSimilaritypoly} and Lemma $2.5$,  \cite{SSS20};
		\item if $X$ is irreducible, then $\rho(X)=\min\set{\Norm{S^{-1}XS} \, : \, S\in \GL_n(\C)}$, follows from Lemma $2.4$, \cite{SSS20};
		\item if $X$ is irreducible and $\rho(X) = 1$, then $X$ is similar to a row coisometry, from Lemma $2.9$. \cite{SSS20}. 
   \end{enumerate}

	We remark in passing that the last item is proved using the so-called ``quantum Perron-Frobenius theorem'' of Evans and H{\o}egh-Krohn \cite{EvansAndHK}.

	\begin{lemma}\label{M is contraction}
   If $A = (A_1,\dots,A_d)$ is irreducible with outer spectral radius $\rho(A)\leq 1$ then $A$ is similar to a column contraction. 
    \end{lemma}
    \begin{proof}
    As $A$ is irreducible then $A^*$ is irreducible as well. Also by definition of spectral radius, $\rho(A^*)=\rho(A)\leq 1$. Then by properties $(3)$ and $(5)$ above, we have that $A^*$ is similar to a row contraction, say $X$. That is, there exists an invertible, $S$ such that $A^* = SXS^{-1}$, taking adjoints gives $A=S^{-*}X^*S^*$. If $X$ is a row contraction then $X^*$ is a column contraction, and so we see that $A$ is similar to a column contraction. 
    \end{proof}
We also require the following lemma from \cite{jury2020noncommutative}:
\begin{lemma} \cite[Proposition 4.1]{jury2020noncommutative}\label{lem:radius-1-no-zeroes}
A monic linear pencil $I-Ax$ is nonsingular in the row ball if and only if $\rho(A)\leq 1$. 
\end{lemma}
%


\section{Proof of Theorem~\ref{mainRODthm-repeat} \label{sec:proof-of-main}}

The proof of the theorem consists of first proving it in the special case where $F$ is a contractive, monic linear pencil (accomplished by the following lemma), then reducing the general case to that one by means of the algebraic and analytic machinery of the previous section, and finishing by induction. 

Throughout the proof, for a matrix polynomial $P(x)=\sum_{w\in \fax} P_w x^w$ we will write $\|P\|_2$ for the Hilbert space norm $(\sum_{w\in\fax} tr(P_w^*P_w))^{1/2}$ and $\|P\|_{mult}$ for the norm of $P$ as a left multiplier $F\to PF$. 

\begin{lemma}\label{The Lemma}
	Assume that $M$ is column contraction: $\|M\|_{{col}}\leq 1$ and let $\pi_n$ be the one variable optimal polynomial approximants for $f=1-z$. Then
	\begin{enumerate}[(a)]
		\item \(\norm{\pi_n(Mx)}_{mult} \lesssim n\), and \label{infnorm}
		\item \(\norm{\pi_n(Mx)(I-Mx)-I}^2_2 \lesssim \frac{1}{n}.\) \label{twonorm}
	\end{enumerate}	
\end{lemma}
\begin{proof}
	We begin by recalling the one variable commutative case \(f(z)=1-z\) where the \(H^2(\D)\) optimal polynomial approximants are given in~\cite{FMS14} by
	\begin{equation}
		\pi_n(z) = \sum_{k=0}^n \brkt{1-\frac{k+1}{n+2}}z^k
	\end{equation}
	To show (\ref{infnorm}), observe that
	\begin{equation}
		\norm{\pi_n}_\infty := \sup\{|\pi_n(z)|:|z|<1\}= \sum_{k=0}^n \brkt{1-\frac{k+1}{n+2}} = \frac{n+1}{2}. \label{FMSsupnorm}
	\end{equation}
	Because \(M\) is a column contraction, by Lemma~\ref{lem:contractive-pencil} we have that multiplication by \(Mx\) is a contraction, so we can apply von Neumann's inequality to see that
	\begin{equation*}
		\norm{\pi_n(Mx)}_{mult} \leq \norm{\pi_n}_\infty \lesssim n
	\end{equation*}
	as needed.

	For part (\ref{twonorm}), again consider the one variable case and use algebra to see that
	\begin{align}
		\pi_n(z)(1-z)-1 
			&= (1-z)\sum_{k=0}^n \brkt{1-\frac{k+1}{n+2}}z^k -1 \nonumber \\
			&= -\frac{1}{n+2} \sum_{k=1}^{n+1} z^k . \label{1voptimalnorm}
	\end{align}
Again, since multiplication by $Mx$ is contractive, we have $\|(Mx)^k\|_2\leq \|Mx\|_2:=c$ for all $k\geq 1$. Moreover, since the \((Mx)^k\) are orthogonal (they are homogeneous polynomials of different degrees), by the Pythagorean theorem we have
	\begin{align}
		\norm{\pi_n(Mx)(I-Mx)-I}_2^2
			&= \norm{-\frac{1}{n+2} \sum_{k=1}^{n+1} (Mx)^k}_2^2  \nonumber \\
			&= \frac{1}{(n+2)^2} \sum_{k=1}^{n+1} \norm{(Mx)^k}_2^2 \leq \frac{c^2(n+1)}{(n+2)^2}, \label{MXoptimalnorm}
	\end{align}
	which gives \(\norm{\pi_n(Mx)(I-Mx)-I}^2_2 \lesssim \frac {1}{n}\).
\end{proof}
Let us now prove Theorem~\ref{mainRODthm-repeat}. We recall the statement:

	\begin{theorem*}(Theorem \ref{mainRODthm-repeat})
    Let $F\in M_k(\C)\otimes \C\fralg{\bbx}$ be a $k\times k$-matrix valued free polynomial, let $(P_n)$ denote the sequence of optimal polynomial approximants, and put 
    \[
    c_n:=\|P_n(x)F(x)-I\|_2^2.
    \]
    Then: \begin{enumerate}
        \item $c_n\to 0$ if and only if $F$ is nonsingular in the row ball.
        \item If $F$ is nonsingular in the row ball, then 
        \[
        c_n\lesssim \frac{1}{n^p}
        \]
        for some $p>0$ depending on $F$. In particular, if $F$ is a product of $\ell$ atomic factors, we can take $p=\frac{1}{3^{\ell-1}}$.
    \end{enumerate}
    \end{theorem*}
    

\begin{proof}[Proof of Theorem~\ref{mainRODthm-repeat}]
By Theorem~\ref{nonsingular-is-necessary} (and the remarks in Section~\ref{sec:matrix-coefficients}), if $c_n\to 0$ then necessarily $F$ is nonsingular in the row ball. The other direction of item (1) will evidently follow from item (2). 
Let 
$F\in M_k(\C\fralg{\bbx})$ and assume it is nonsingular on $\mfB^d$.
We can assume, without loss of generality, that $F(0)=I$. 
Then by lemma \ref{linearization trick} we have that $F$ is stably associated to a monic linear pencil, say, $L_A(x)$ and 
\[
L_A(x)\sim\begin{bmatrix}
 L_1(x) & * & * & \hdots & * \\
 0 & L_2(x) & * & \hdots & * \\
 0 & 0 & L_3(x) &  \hdots & * \\
 \vdots & \vdots &\vdots & \ddots  & \vdots\\
 0 & 0 & 0 & \hdots & L_l(x)
\end{bmatrix}
\]
where for each $k$ the linear pencil $L_k(x)=I-\Bar{A}^{(k)}x$ has $\Bar{A}^{(k)}$ either $0$ or irreducible. 
By hypothesis, our original $F$ was nonsingular in the row ball, and since stable associativity preserves the zero locus, the pencil $L_A$, and hence also each of the $L_{\Bar{A}^{(k)}}$, is also nonsingular in the row ball. Thus by Lemma~\ref{lem:radius-1-no-zeroes}, each of the pencils $\Bar{A}^{(k)}$ has outer spectral radius at most $1$. It follows, in turn, by Lemma~\ref{M is contraction}, that each $\Bar{A}^{(k)}$ is similar to a column contractive pencil $M^{(k)}$.   



In summary, we conclude that that our $F$ is stably associated to a monic linear pencil of the following block upper-triangular form:

\begin{equation}\label{eqn:main-reduction-step}
F\sim_{\mathrm{sta}}\begin{bmatrix}
 I-M^{(1)}(x) & * & * & \hdots & * \\
 0 & I-M^{(2)}(x) & * & \hdots & * \\
 0 & 0 & I-M^{(3)}(x) &  \hdots & * \\
 \vdots & \vdots &\vdots & \ddots  & \vdots\\
 0 & 0 & 0 & \hdots & I-M^{(l)}(x)
\end{bmatrix}
\end{equation}
where each $M^{(k)}$ is a column contraction. Note that the off-diagonal entries are arbitrary linear pencils $Y_{ij}(x)$ (genuinely linear, with no constant term). \\

By Remark $5.2$ and Section $6.2$ of \cite{HKVfactor}, the number of blocks $\ell$ will be equal to the number of atomic factors of $F$. 

We now consider a monic linear pencil $L_A(x)$ of the form appearing in the right-hand side of (\ref{eqn:main-reduction-step}), in $\ell\times \ell$ block upper triangular form, with contractive pencils down the diagonal. We prove the following

{\bf Claim:} {\em for each such pencil $L_A(x)$, there exists a sequence of matrix polynomials $\sigma_{n,\ell}(x)$, $n=1, 2, \dots, $ satisfying the following conditions: 
\begin{enumerate}
    \item $\deg \sigma_{n,\ell} \leq (\ell-1) +n +n^3+ \cdots +n^{3^{\ell-1}} \lesssim n^{3^{\ell-1}}$, 
    \item $\|\sigma_{n,\ell}\|_{mult}\lesssim n^{(1+3+3^2+\dots+3^{l-1})}$, and 
    \item $\|\sigma_{n,\ell}(x) (I-Ax)-I\|_2^2 \lesssim \frac1n$
\end{enumerate}
for all $n$, where the implied constants are allowed to depend on $\ell$ but not on $n$. }

Assuming the claim for the moment, we conclude that for this fixed $A$ (which fixes $\ell$) there exist, for each $n$, matrix polynomials $\sigma_N:=\sigma_{n, \ell}$ of degree $N\lesssim n^{3^{\ell-1}}$ such that
\[
\|\sigma_N(x) (I-Ax)-I\|_2^2 \lesssim \frac1n
\]
It follows that for the optimal approximants $P_N$ at this same degree $N$, we will also have $\| Q_N(x)(I-Ax)-I\|_2^2\lesssim \frac1n\lesssim \frac{1}{N^p}$ (where $p=\frac{1}{3^{\ell-1}}$), and thus (since these quantities decrease as the degree $N$ increases), we conclude that for all degrees $n$
\[
\|Q_n(x) (I-Ax)-I\|_2^2 \lesssim \frac{1}{n^p}
\]
where $p=\frac{1}{3^{\ell-1}}$. Finally, since $I-Ax$ was stably equivalent to our original matrix polynomial $F$, we conclude by Lemma \ref{same ROD} that the optimal approximants $P_n$ of the original $F$ achieve the same rate of decay. 

To complete the proof, it remains to prove the claim, which we will do via induction on $\ell$. 

In the case $\ell=1$, our pencil $L_A(x)$ has the form $1-M^{(1)}x$ for some column contractive irreducible pencil $M^{(1)}$, so we can take $\sigma_{n,1}$ to be the polynomials provided by Lemma \ref{The Lemma}, namely $\sigma_{n,1}(x) =\pi_n(M^{(1)}x)$ where $\pi_n$ are the $1$-variable OPAs for the polynomial $f(z)=1-z$. 

Now suppose the claim is proved for $\ell$, let us prove it for $\ell+1$. We write our pencil in the form
\begin{equation*}
    L_A(x) = \begin{bmatrix} I-\widetilde{A}x &  Yx \\ 0 & I-M^{(\ell+1)}x\end{bmatrix}
\end{equation*}
where $\widetilde{A}$ has at most $\ell$ irreducible blocks. 
We will choose $\sigma_{n,\ell+1}$ of the form
\[
\sigma_{n, \ell+1} (x) =\begin{bmatrix}\sigma_{n,\ell}(x) & r(x) \\ 0 & q(x)\end{bmatrix}
\]
where $\sigma_{n,\ell}$ is chosen by applying the induction hypothesis to the block $\ell\times\ell$ pencil $I-\widetilde{A}x$. 

We choose $q(x)$ to be $\pi_n(M^{(\ell+1)}(x))$, as in Lemma \ref{The Lemma}, so $\deg(q)\leq n$ and $\|q\|_{mult}\lesssim n$ by the lemma. 

Finally we define
\[
r(x) = -\sigma_{n,\ell}(x) \cdot (Yx) \cdot \pi_{N}(M^{(\ell+1)}x), 
\]
where again $\pi_N$ is the one-variable OPA for $f(z)=1-z$, at degree $N=n^{3^\ell}$. Let us now verifty the claims (1)-(3). First, the degree of $\sigma_{n, \ell+1}$ will be the maximum of the degrees of $\sigma_{n,\ell}, r(x)$, and $q(x)$, and we see by inspection that $r(x)$ has the largest degree, which is
\begin{align}
\deg r(x) &= \deg(\sigma_{n, \ell}) + 1 + n^{3^\ell}\\
            &\leq \ell + n+ n^3 + \cdots + n^{3^{\ell}}.
\end{align}
which proves (1). 
For (2), observe that the $\|\cdot\|_{mult}$ norm of $\sigma_{n, \ell+1}$ is comparable to the maximum of the $\|\cdot\|_{mult}$ norms of its entries; by the induction hypothesis, Lemma \ref{The Lemma}, and the definition of $r(x)$, we have
\begin{align*}
\|r(x)\|_{mult} &\leq \|\sigma_{n, \ell}\|_{mult} \|Yx\|_{mult} \|\pi_N(M^{(\ell+1)}x)\|_{mult} \\
&\lesssim n^{1+3+\cdots +3^{\ell}}.
\end{align*}
Since this is larger than $\|\sigma_{n, \ell}\|_{mult}$ and $\|q(x)\|_{mult}$, (2) is proved. 

Finally for (3), we have 

\begin{equation*}
\sigma_{n, \ell+1}(x) (I-A^{(\ell+1)}x) = \begin{bmatrix} \sigma_{n, \ell}(x) (I-A^{(\ell)}x) & \sigma_{n, \ell}(x)\cdot Yx + r(x) (I-M^{(\ell+1)}x) \\ 0 & q(x) (I-M^{(\ell+1)}x)\end{bmatrix}
\end{equation*}

Subtracting $\begin{bmatrix} I & 0 \\ 0 & I\end{bmatrix}$ and taking the (squared) Hilbert space norm, in the diagonal entries we have by the induction hypothesis and the Lemma
\[
\| \sigma_{n, \ell}(x) (I-A^{(\ell)}x)-I \|_2^2 \lesssim \frac1n
\]
and 
\[
\|q(x) (I-M^{(\ell+1)}x) -I\|_2^2 = \|\pi_{n}(M^{(\ell+1)}x) (I-M^{(\ell+1)}x) -I\|_2^2 \lesssim \frac1n.
\]
In the off-diagonal entry, we have by the definition of $r(x)$, the induction hypothesis, and the Lemma (applied at degree $N=n^{3^\ell}$), 
\begin{align*}
\|\sigma_{n, \ell}(x)\cdot Yx + &r(x) (I-M^{(\ell+1)}x) \|_2^2 \leq \\ &\leq\|\sigma_{n,\ell}\|_{mult}^2 \|\|Yx\|_{mult}^2 \|I - \pi_N(M^{(\ell+1)}x) (I-M^{(\ell+1)}x\|_2^2 \\
&\lesssim n^{2(1+3+\cdots 3^{(\ell-1)})} \cdot \frac{1}{n^{3^\ell}}\\
&= n^{3^\ell -1} \cdot \frac{1}{n^{3^\ell}}\\
&=\frac1n.
\end{align*}
This completes the proof of the claim. 
\end{proof}

As an immediate consequence of the theorem (and the remarks following Definition~\ref{def:free-cyclic}), we obtain a new proof of the following fact, first proved in \cite{jury2020noncommutative}: 

\begin{corollary} Let $f$ be an nc polynomial with scalar coefficients. If $f$ is nonvanishing in the row ball, then $f$ is cyclic for the $d$-shift. 
\end{corollary}

\subsection{Remarks and Questions} In general, the exponent $p=\frac{1}{3^{\ell-1}}$ appearing in Theorem~\ref{mainRODthm-repeat} (where $\ell$ is the number of irreducible factors of $F$) is not sharp. Already in one variable, the bound $c_n\lesssim \frac{1}{n}$ holds regardless of the number of factors, as was already mentioned in the introduction.

It is possible to run the proof presented here in the one-variable case, making more careful choices of the polynomials constructed in the induction step, to recover the uniform $\frac{1}{n}$ bound in that case, but this relies crucially on the commutativity of the polynomial ring $\C [z]$.

At the moment, we do not know if the same is true in the free setting, or even if there is a uniform power $p$ so that $\|P_nF-I\|_2^2\lesssim \frac{1}{n^p}$ independently of $F$.  Also, observe that our proof does not actually construct the optimal approximants for any given $f$, we only construct a sequence $p_n$ obeying the claimed bounds, so that the optimal approximants must be at least as good. We have computed the optimal approximants (with computer assistance) in some simple examples, but the results do not appear enlightening and we have not included them.  

The sharp value of $p$ seems difficult to calculate even for relatively simple cases. For example, for the free polynomial in two variables $f(x,y)=(1-x)(1-y)$ one can construct polynomials $p_n$ (by essentially the method used in the induction step of the proof, but making more careful choices of $\sigma_n$, $r_n$) for which
\[
\norm{ p_n(x,y)(1-x)(1-y)-1}_2^2\lesssim \frac{1}{n^{1/2}},
\]
which improves on the bound $\frac{1}{n^{1/3}}$ provided by the theorem, but we do not know if this value $p=1/2$ is sharp. 

Finally, it would be of interest to know 1) whether the methods used here can be modified to handle the case of nc rational functions $F$, and 2) what can be said in the case of non-square polynomial matrix functions $F$. 

\bibliographystyle{acm} 
\bibliography{bibliography.bib}

\end{document}

%% file: newcommands.tex
\newcommand{\abs}[1]{{\vert #1 \vert}}

\newcommand{\norm}[1]{{ \left\Vert #1 \right\Vert }}
\newcommand{\Norm}[1]{{ \left\Vert #1 \right\Vert }}
\newcommand{\set}[1]{{ \left\{ #1 \right\} }}
\newcommand{\brkt}[1]{{\left( #1 \right)}}

\newcommand{\cF}{ {\mathcal{F}} }

\newcommand{\sZ}{ {\mathscr{Z}} }

\newcommand{\mfB}{ {\mathfrak{B}} }

\newcommand{\C}{{\mathbb C}}

\newcommand{\D}{{\mathbb D}}
\newcommand{\RR}{{\mathbb R}}
\newcommand{\ZZ}{{\mathbb Z}}
\newcommand{\NN}{{\mathbb N}}

\newcommand{\bbx}{\mathbbm{x}}

\newcommand{\x}{\bm{x}}
\newcommand{\y}{\bm{y}}

\newcommand{\fralg}[1]{\langle #1 \rangle}

\newcommand{\fax}{\langle \bbx \rangle}

\makeatletter
\def\moverlay{\mathpalette\mov@rlay}
\def\mov@rlay#1#2{\leavevmode\vtop{
                \baselineskip\z@skip \lineskiplimit-\maxdimen
                \ialign{\hfil$#1##$\hfil\cr#2\crcr}}}
\makeatother

\makeatletter
\DeclareFontFamily{OMX}{MnSymbolE}{}
\DeclareSymbolFont{MnLargeSymbols}{OMX}{MnSymbolE}{m}{n}
\SetSymbolFont{MnLargeSymbols}{bold}{OMX}{MnSymbolE}{b}{n}
\DeclareFontShape{OMX}{MnSymbolE}{m}{n}{
    <-6>  MnSymbolE5
   <6-7>  MnSymbolE6
   <7-8>  MnSymbolE7
   <8-9>  MnSymbolE8
   <9-10> MnSymbolE9
  <10-12> MnSymbolE10
  <12->   MnSymbolE12
}{}
\DeclareFontShape{OMX}{MnSymbolE}{b}{n}{
    <-6>  MnSymbolE-Bold5
   <6-7>  MnSymbolE-Bold6
   <7-8>  MnSymbolE-Bold7
   <8-9>  MnSymbolE-Bold8
   <9-10> MnSymbolE-Bold9
  <10-12> MnSymbolE-Bold10
  <12->   MnSymbolE-Bold12
}{}

\let\llangle\@undefined
\let\rrangle\@undefined
\DeclareMathDelimiter{\llangle}{\mathopen}%
                     {MnLargeSymbols}{'164}{MnLargeSymbols}{'164}
\DeclareMathDelimiter{\rrangle}{\mathclose}%
                     {MnLargeSymbols}{'171}{MnLargeSymbols}{'171}
\makeatother

\newcommand{\GL}{\mathrm{GL}}

\newcommand{\ip}[2]{\langle #1, \, #2 \rangle}

\mathchardef\mhyphen="2D

\newcommand{\bsbm}{\left[ \begin{smallmatrix}}
\newcommand{\esbm}{\end{smallmatrix} \right]}

\newcommand{\bbm}{ \begin{bmatrix}}
\newcommand{\ebm}{\end{bmatrix} }

\newcommand{\bpm}{\begin{pmatrix}}
\newcommand{\epm}{\end{pmatrix}}

\newcommand{\bspm}{\left(\begin{smallmatrix}}
\newcommand{\espm}{\end{smallmatrix}\right)}

\newcommand{\bsm}{\begin{smallmatrix}}
\newcommand{\esm}{\end{smallmatrix}}

\newcommand{\bal}{\begin{align*}}
\newcommand{\eal}{\end{align*}}
